\documentclass[12pt,a4paper]{article}

\usepackage{amsmath, amssymb, amsthm, verbatim} 
\usepackage{graphicx}

\usepackage{bm}
\usepackage{eucal} 

\theoremstyle{plain}
  \newtheorem{theorem}{Theorem}[section]
  \newtheorem{prop}[theorem]{Proposition}
  \newtheorem{lemma}[theorem]{Lemma}
   
  \newtheorem{assumption}[theorem]{Assumption}
  
  \newtheorem{definition}[theorem]{Definition}
  \newtheorem{dflem}[theorem]{Definition-Lemma}
  \newtheorem{remark}[theorem]{Remark} 
  \newtheorem{example}[theorem]{Example}

\newcommand{\cl}{\operatorname{cl}}

\newcommand{\Real}{\operatorname{Re}}

\newcommand{\A}{\mathcal{A}}
\newcommand{\F}{\mathcal{F}}
\newcommand{\ch}{\mathsf{ch}}
\newcommand{\M}{\mathsf{M}}
\newcommand{\Sep}{\operatorname{Sep}}

\title{Semi-algebraic partition and basis 
of Borel-Moore homology of hyperplane arrangements}
\author{Ko-Ki Ito, Masahiko Yoshinaga}
\date{\today}

\begin{document} 

\maketitle
\begin{abstract}
We describe an explicit semi-algebraic partition 
for the complement of a real hyperplane arrangement such 
that each piece is contractible and so that the pieces form a basis 
of Borel-Moore homology. 
We also give an explicit correspondence between the 
de Rham cohomology and the Borel-Moore homology. 
\end{abstract}

\section{Introduction}
\label{sec:intro}

Semi-algebraic partitions of algebraic 
sets have been studied in various fields of 
mathematics from geometry to computational 
algebra. General theory says that 
any algebraic set has a semi-algebraic triangulation. 
However, not only triangulations, but also 
other types of ``efficient'' decompositions are 
sometimes useful. 
The following are motivating 
examples. 
\begin{example}
$\Bbb{CP}^n=\Bbb{C}^n\sqcup\Bbb{C}^{n-1}\sqcup\dots
\sqcup\Bbb{C}^1\sqcup\{pt\}$. 
\end{example}

\begin{example}
\label{ex:mot}
\label{ex:c*}
Consider $\Bbb{C}^*=\Bbb{C}\setminus\{0\}$. 
Put 
\begin{eqnarray*}
S_0&=&\{re^{i\theta}\mid r>0, 0<\theta<2\pi\}\\
S_1&=&\{z\mid z\in\Bbb{R}, z>0\}. 
\end{eqnarray*}
Then $\Bbb{C}^*=S_0\sqcup S_1$.
\end{example}

Both of these decompositions 
reflect (co)homological 
structures of the manifolds naturally. More precisely, 
they are 
\begin{itemize}
\item[(i)] disjoint unions of contractible semi-algebraic subsets, and 
\item[(ii)] the closures of the pieces form a basis of 
Borel-Moore (locally finite) homology 
$H^{BM}_*(X, \Bbb{Z})$ (or ordinary cohomology 
$H^*(X, \Bbb{Z})$ via Poincar\'e duality.) 
\end{itemize}

The purpose of this paper is to generalize 
Example \ref{ex:mot} to hyperplane 
arrangements defined over the real number field $\Bbb{R}$. 

There is another reason to expect the existence of
such partitions. The complement of a complex hyperplane
arrangement is known to be a minimal space
\cite{dim-pap, fal-hom, ran-mor}, that is
a space homotopy equivalent to a finite CW complex with
exactly as many
$k$-cells as the $k$-th Betti number, for all $k$. 
If the arrangement is defined over $\Bbb{R}$,
then the real structures (e.g., chambers) are
related to the topology of the complexified
complement (\cite{sal-top, zas-face}).
With the help of real structures, the minimal CW
complex has been explicitly described in
\cite{sal-sett} and \cite{Y}. It is natural to
expect that the dual stratification to
a minimal CW complex induces a partition
satisfying (i) and (ii) above.

\section{Preliminary}
\label{sec:prel}

Let 
$\A = \{ H_1 , \ldots , H_n \}$ 
be an affine hyperplane arrangement 
in the real vector space $\Bbb{R} ^\ell$. 
Let us fix a defining affine linear form $\alpha_i$ 
in such a way that $H_i=\{\alpha_i=0\}$. 
Let $L(\A)$ 
be the set of nonempty intersections of elements of $\A$. 
We denote by $L^p(\A)$ the set of all $p$-dimensional 
affine subspaces $X\in L(\A)$. 

For an affine subspace $X\subset\Bbb{R}^\ell$, let us denote 
$X_\Bbb{C}=X\otimes_\Bbb{R}\Bbb{C}$ the complexification. 
Denote by 
$\ch (\A) $
the set of all chambers and by 
$\M ( \A) 
= 
\Bbb{C} ^\ell \setminus 
\bigcup _{H\in\A} H_\Bbb{C}
$
the complement to the complexified hyperplanes.

Let $\F$ be a generic flag 
 in $\Bbb{R} ^\ell$ 
$$
\F : 
\emptyset = \F ^{-1} \subset 
\F ^0 
\subset 
\F ^1 
\subset 
\cdots 
\subset 
\F ^\ell = \Bbb{R} ^\ell ,  
$$  
where  
each $\F ^q$ 
is a {\it generic} $q$-dimensional affine subspace, 
that is, 
$\dim \F ^q \cap X = q+ \dim X-\ell$ 
for $X \in L (\A)$. 
Let 
$\{ h_1 , \ldots h_\ell \}$ 
be a system of defining equations of $\F$, 
that is, 
$$
\F ^{q} 
= \{ h_{q+1} = \cdots = h_\ell =0 \} ,  \mbox{ for }
q=0, 1, \dots, \ell-1, 
$$  
where each $h_i$ is an affine linear form on $\Bbb{R} ^\ell$. 
Define 
$$ 
\ch ^q _\F ( \A ) 
= \{ 
C \in \ch (\A) \mid 
C \cap \F ^q \not= \emptyset \quad \mbox{and} \quad 
C \cap \F ^{q-1} = \emptyset 
\} . 
$$ 
We assume that $\F$ satisfies the following : 
\begin{assumption} 
\label{assum:h}
For $q=0, \ldots ,\ell$, 
$\F ^{q} _{>0}$ 
 denotes  
$$ 
\{ h_{q+1}= h_{q+2} = \cdots = h_\ell =0, h_{q} >0 
\} . 
$$  
\begin{enumerate} 
\item For an arbitrary chamber $C$, 
if belonging to  
$\ch ^{q} _{\F} (\A)$, then 
$C \cap \F ^{q} \subset \F ^{q} _{>0}$.     
\item For any two 
$X$, $X^\prime \in L(\A)$ with 
$\dim X=\dim X^\prime = \ell- q$ (i.e. 
satisfying 
$X \cap \F^{q} = \{ 1pt \}$ 
and 
$X ^\prime \cap \F^{q} = \{ 1pt \}$),  
if $X \not= X^\prime$, then 
$$ 
h_{q} (X \cap \F^{q}) 
\not= 
h_{q} (X ^\prime \cap \F^{q})  . 
$$ 
\end{enumerate} 
\end{assumption} 
Note that such a flag always exists. Indeed, we first 
choose a generic hyperplane $\F^{\ell-1}$ in such a way 
that $\F^{\ell-1}$ does not separate $0$-dimensional 
intersections $L^0(\A)$. 
In a similar fashion, we choose 
$\F^{\ell-2}\subset\F^{\ell-1}$ inductively. 

Let us denote $\M^q:=(\F^q_\Bbb{C})\cap\M(\A)=
\F_\Bbb{C}^q\setminus\bigcup_{H\in\A}H_\Bbb{C}$. 
We have the following propositions (see \cite{orl-ter, Y2}). 
\begin{prop}
\label{prop:trunc}
Let $\A$ be an arrangement and $W^q$ a 
$q$-dimensional generic subspace. 
Let $\A\cap W^q$ be the arrangement on $W^q$ induced by $\A$. 
\begin{itemize}
\item[$(1)$] 
Then $L(\A\cap W^q)$ is 
isomorphic to $L^{\geq \ell-q}(\A)=\{X\in L(\A)\mid 
\dim X\geq\ell-q\}$ as posets. 
\item[$(2)$] 
Then the natural inclusion 
$i:\M(\A)\cap W^q\hookrightarrow \M(\A)$ 
induces isomorphisms 
$$
i_k: H_k(\M(\A)\cap W^q, \Bbb{Z})\stackrel{\cong}{\longrightarrow}
H_k(\M(\A), \Bbb{Z}), 
$$
for $k=0, 1, \ldots, q$. 
\end{itemize}
\end{prop}

\begin{prop}
\label{prop:betti}
Let $\A$ be a real arrangement and $\F$ a generic flag. 
Then $|\ch_\F^q(\A)|=b_q(\M(\A)) $. 
\end{prop}

Let $X, Y\subset\Bbb{C}^\ell$ be subsets. 
We denote by $\overline{Y}$ the closure 
of $Y$ in $\Bbb{C}^\ell$ (with respect to the 
classical topology) and $\cl_X(Y)=X\cap\overline{Y}$. 

For given $C\in\ch_\F^q(\A)$, 
$\cl_{\F^q}(C)=\overline{C}\cap\F^q\subset\Bbb{R}^\ell$ 
is a $q$-dimensional polytope 
which does not intersect $\F^{q-1}$. 
By Assumption \ref{assum:h}, the vertices 
of $\cl_{\F^q}(C)$ have mutually different and 
positive heights 
with respect to $h_q$ (noting that 
$\F^{q-1}=\F^q\cap\{h_q=0\}$). 
There is a unique vertex $p\in\cl_{\F^q}(C)$ at which 
$h_q|_{\cl_{\F^q}(C)}$ attains the minimum. 
Then by Proposition \ref{prop:trunc} (1), 
there exists unique intersection $X_C\in L(\A)$ 
satisfying $\{p\}=X_C\cap\F^q$. (Note that 
in case $C\in\ch_\F^0(\A)$, we consider $X_C=\Bbb{R}^\ell$.) 

\begin{definition}
Let $X\subset\Bbb{R}^\ell$ be an affine subspace. 
Denote by $\tau(X)$ the linear subspace through the 
origin which is parallel to $X$ and $\dim \tau(X)=\dim X$. 
And define 
$$
\A_{[X]}:=
\{H\in\A\mid
\tau(H)\supset\tau(X)\} 
$$
the set of hyperplanes parallel to $X$. Note that 
$\A_X:=\{H\in\A\mid H\supset X\}\subseteq
\A_{[X]}$. 
\end{definition}

\begin{definition}
Let $C\in\ch_\F^q(\A)$. 
We denote by $\widetilde{C}\in\ch(\A_{[X_C]})$ 
the unique chamber with $C\subset
\widetilde{C}$. 
\end{definition}
Using these notations, 
we shall define a partial ordering 
$\preceq$ in $\ch^q_\F(\A)$. 
\begin{definition}
\label{def:order}
For $C, C'\in\ch^q_\F(\A)$, we denote 
$C\preceq C'$ if and only if $C'\subset\widetilde{C}$. 
\end{definition}
The following is easy. 

\begin{lemma}
\label{lem:order}
If $C\preceq C'$, then $h_q(X_C\cap\F^q)\leq 
h_q(X_{C'}\cap\F^q)$. 
\end{lemma}

\section{Minimal partition}

In this section, we shall introduce 
the semialgebraic partition. 

\begin{definition}
Let $p_1, p_2\in\Bbb{R}^\ell$. 
The set $\Sep(p_1, p_2)$ of separating hyperplanes 
is defined by 
$$
\Sep(p_1, p_2):=\{
H \in \A \mid [p_1, p_2] \cap H \neq \emptyset\}, 
$$
where $[p_1, p_2]$ is the closed line segment 
connecting two points $p_1$ and $p_2$. 

Similarly, we also denote by $\Sep(C_1, C_2)$ 
the set of separating 
hyperplanes of two chambers $C_1, C_2$. 
\end{definition}

\begin{lemma}
\label{lem:key1}
Let $C, C'\in\ch^q_\F(\A)$. If $X_C=X_{C'}$, then 
$\Sep(C, C')\subset\A_{X_C}$. 
\end{lemma}

\begin{proof}
Let $H\in\Sep(C, C')$ and choose a defining 
equation $f$, i.e., $H=\{f=0\}$. 
Since $H$ separates $C$ and $C'$, we may assume 
$C\subset \{f\geq 0\}$ and $C'\subset\{f\leq 0\}$. 
Hence 
$X_C\cap\F^q\subset \cl_{\F^q}(C)\subset\{f\geq 0\}$. 
Similarly, 
$X_{C'}\cap\F^q\subset \cl_{\F^q}(C')\subset\{f\leq 0\}$. 
We have 
$X_C\cap\F^q\subset\{f\geq 0\}\cap\{f\leq 0\}\cap\F^q
=H\cap\F^q$. Then the inclusion $X_C\subset H$ 
follows from Proposition \ref{prop:trunc} (1). 
This means that $H\in\A_{X_C}$. 
\end{proof}

\begin{lemma}
\label{lem:key2}
Let $C\in\ch_\F^{q}(\A)$ and 
$C'\in\ch_\F^{q'}(\A)$. 
\begin{itemize}
\item[(1)] If $C'\subset\widetilde{C}$, then 
either $C\preceq C'$ (with $q=q'$), or $q<q'$. 
\item[(2)] If $q=q'$ and 
$h_q(X_C\cap\F^q)< h_q(X_{C'}\cap\F^q)$, then 
$\A_{X_{C'}}\cap\Sep(C, C')\neq\emptyset$. 
\item[(3)] 
If 
$q<q'$, 
then $\A_{X_{C'}}\cap\Sep(C, C')\neq\emptyset$. 
\end{itemize}
\end{lemma}

\begin{proof}
(1) 
First note that 
$\widetilde{C}\cap\F^{q-1}=
(\widetilde{C}\cap\F^q)\cap\F^{q-1}$. 
Since $h_q|_{\widetilde{C}\cap\F^q}$ attains 
the minimum at $X_C\cap\F^q$, 
$h_q|_{\widetilde{C}\cap\F^q}>0$. 
Hence $C'\cap\F^{q-1}=\emptyset$ and we have $q'\geq q$. 
The assertions thus follow from Definition 
\ref{def:order}. 

(2) 
Suppose that 
$\A_{X_{C'}}\cap\Sep(C, C')=\emptyset$. 
Then $C$ and $C'$ are contained in the same 
chamber $D\in\ch(\A_{X_{C'}})$ 
of $\A_{X_{C'}}$. Since $h_q|_{\cl_{\F^q}(D)}$ attains 
the minimum at $X_{C'}\cap\F^q$, we have 
$h_q(X_C\cap\F^q)\geq h_q(X_{C'}\cap\F^q)$. This 
contradicts the assumption. 

(3) 
Suppose that 
$\A_{X_{C'}}\cap\Sep(C, C')=\emptyset$. 
Then $C$ and $C'$ are contained in the same 
chamber $D\in\ch(\A_{X_{C'}})$ 
of $\A_{X_{C'}}$. Obviously 
$D\cap\F^{q'-1}=\emptyset$ and 
hence $C\cap\F^{q}=\emptyset$. This contradicts 
to $C\in\ch_\F^q(\A)$. 
\end{proof}

From now on we fix a base point $p_C\in C\cap\F^q$ 
for each $C\in\ch_\F^q(\A)$. 
It is easily seen that the constructions below do not 
depend on the choice of $p_C$. 

We can identify $\Bbb{C}^\ell$ with the tangent 
bundle 
$T\Bbb{R}^\ell\cong\Bbb{R}^\ell\times \Bbb{R}^\ell$ by 
$$
\begin{array}{ccl}
\Bbb{R}^\ell\times \Bbb{R}^\ell&\longrightarrow&\Bbb{C}^\ell\\
(x, v)&\longmapsto&x+\sqrt{-1}v. 
\end{array}
$$
We also denote $x$ by $\Real(x+\sqrt{-1}v)$. 

Now we introduce the main object of this paper. 
\begin{definition}
For a chamber $C\in\ch(\A)$, we define 
$$
S(C)=\left\{
x+\sqrt{-1}v\in\Bbb{C}^\ell \left|
\begin{array}{l}
v\in\tau(X_C), x\in \Bbb{R}^\ell \mbox{ and}\\
v\notin \tau(H), \mbox{ for } H\in\Sep(p_C, x)
\end{array}
\right.
\right\}. 
$$
\end{definition}
If $C\in\ch_\F^q(\A)$, $S(C)$ is an open 
subset of $\Bbb{R}^\ell\times\tau(X_C)$, hence 
real $(2\ell-q)$-dimensional 
manifold. 
\begin{example}
\label{ex:1dim}
Let $H=\{0\}\subset\Bbb{R}$ and the arrangement 
$\A=\{H\}$. Fix a generic flag 
$\F^{0}=\{-1\}$. There are two chambers 
$C_0=\Bbb{R}_{<0}$ and $C_1=\Bbb{R}_{>0}$. 
Then $\ch_\F^0(\A)=\{C_0\}$ and 
$\ch_\F^1(\A)=\{C_1\}$. 
Then $S(C_0)=S_0$ and $S(C_1)=S_1$ as 
Example \ref{ex:mot}. 
\end{example}

\begin{example}
\label{ex:3lines}
Let $\A=\{H_1, H_2, H_3\}$ be an arrangement of lines 
on $\Bbb{R}^2$ and 
fix a generic flag $\F^\bullet$ 
as in Figure \ref{fig:ex}. 
Then 
$\ch_\F^0(\A)=\{C_0\}, 
\ch_\F^1(\A)=\{C_1, C_2, C_3\}, 
\ch_\F^2(\A)=\{C_4, C_5\}$. 
By definition, we also have 
$X_{C_0}=\Bbb{R}^2$, 
$X_{C_1}=H_1, X_{C_2}=H_2, X_{C_3}=H_3$, 
and 
$X_{C_4}=H_1\cap H_3$, 
$X_{C_5}=H_1\cap H_2$, 
$\A_{[X_{C_0}]}=\emptyset$, 
$\A_{[X_{C_1}]}=\{H_1\}$, 
$\A_{[X_{C_2}]}=\A_{[X_{C_3}]}=\{H_2, H_3\}$, 
$\A_{[X_{C_4}]}=\A_{[X_{C_5}]}=\A$, 
$\widetilde{C_0}=\Bbb{R}^2$, 
$\widetilde{C_1}=C_1\cup C_2\cup C_3$, 
$\widetilde{C_2}=C_2\cup C_5$, 
$\widetilde{C_3}=C_3\cup C_4$, 
$\widetilde{C_4}=C_4$, 
$\widetilde{C_5}=C_5$. 
Using these data, we can describe 
$S(C)$. For example, 
$S(C_4)=C_4$, $S(C_5)=C_5$. 
Other pieces are shown in 
Figure \ref{fig:ex}. 
(In the figure, a dotted line indicates 
the direction to which $v$ can not be directed.)

\end{example}

\begin{figure}[htbp]
\begin{picture}(100,430)(0,0)
\thicklines

\put(273,303){$H_3$}
\put(233,303){$H_2$}
\put(105,303){$H_1$}

\put(120,310){\line(1,1){120}}
\multiput(230,310)(40,0){2}{\line(-1,1){120}}

\put(50,400){$\F^2=\Bbb{R}^2$}

\multiput(50,330)(5,0){63}{\circle*{2}}
\put(60,335){$\F^0$}
\put(75,330){\circle*{4}}

\put(335,335){$\F^1$}

\put(215,385){$C_3$}
\put(192.5,362.5){$C_2$}
\put(170,340){$C_1$}
\put(150,360){$C_0$}
\put(172.5,382.5){$C_5$}
\put(195,405){$C_4$}

\put(280,410){$\A=\{H_1, H_2, H_3\}$}
\put(280,392){$\ch_\F^0(\A)=\{C_0\}$}
\put(280,375){$\ch_\F^1(\A)=\{C_1, C_2, C_3\}$}
\put(280,358){$\ch_\F^2(\A)=\{C_4, C_5\}$}



\put(20,0){\line(1,1){120}}
\multiput(130,0)(40,0){2}{\line(-1,1){120}}
\put(150,40){$S(C_1)$}

\put(70,10){\circle*{4}}
\put(71.41,11.41){\vector(1,1){10}}
\put(68.59,8.59){\vector(-1,-1){10}}
\put(105,45){\circle{4}}
\put(106.41,46.41){\vector(1,1){10}}
\put(103.59,43.59){\vector(-1,-1){10}}
\put(140,80){\circle{4}}
\put(141.41,81.41){\vector(1,1){10}}
\put(138.59,78.59){\vector(-1,-1){10}}

\put(220,0){\line(1,1){120}}
\multiput(330,0)(40,0){2}{\line(-1,1){120}}
\put(350,40){$S(C_3)$}

\put(340,80){\circle*{4}}
\put(338.59,81.41){\vector(-1,1){10}}
\put(341.41,78.59){\vector(1,-1){10}}
\put(300,120){\circle{4}}
\put(298.59,121.41){\vector(-1,1){10}}
\put(301.41,118.59){\vector(1,-1){10}}


\put(20,130){\line(1,1){120}}
\multiput(130,130)(40,0){2}{\line(-1,1){120}}
\put(150,170){$S(C_0)$}

\put(140,210){\circle{4}}
\put(140,212){\vector(0,1){10}}
\put(140,208){\vector(0,-1){10}}
\put(138,210){\vector(-1,0){10}}
\put(142,210){\vector(1,0){10}}
\multiput(143,213)(2.2,2.2){5}{\circle*{2}}
\multiput(137,207)(-2.2,-2.2){5}{\circle*{2}}
\multiput(137,213)(-2.2,2.2){5}{\circle*{2}}
\multiput(142,208)(2.2,-2.2){5}{\circle*{2}}

\put(105,175){\circle{4}}
\put(105,177){\vector(0,1){10}}
\put(105,173){\vector(0,-1){10}}
\put(103,175){\vector(-1,0){10}}
\put(107,175){\vector(1,0){10}}
\multiput(108,178)(2.2,2.2){5}{\circle*{2}}
\multiput(102,172)(-2.2,-2.2){5}{\circle*{2}}
\multiput(102,178)(-2.2,2.2){5}{\circle*{2}}
\multiput(107,173)(2.2,-2.2){5}{\circle*{2}}

\put(70,140){\circle{4}}
\put(70,142){\vector(0,1){10}}
\put(70,138){\vector(0,-1){10}}
\put(68,140){\vector(-1,0){10}}
\put(72,140){\vector(1,0){10}}
\put(68.59,141.41){\vector(-1,1){10}}
\put(71.41,138.59){\vector(1,-1){10}}
\multiput(73,143)(2.2,2.2){5}{\circle*{2}}
\multiput(67,137)(-2.2,-2.2){5}{\circle*{2}}

\put(30,180){\circle*{4}}
\put(30,182){\vector(0,1){10}}
\put(30,178){\vector(0,-1){10}}
\put(28,180){\vector(-1,0){10}}
\put(32,180){\vector(1,0){10}}
\put(31.41,181.41){\vector(1,1){10}}
\put(28.59,178.59){\vector(-1,-1){10}}
\put(28.59,181.41){\vector(-1,1){10}}
\put(31.41,178.59){\vector(1,-1){10}}

\put(65,215){\circle{4}}
\put(65,217){\vector(0,1){10}}
\put(65,213){\vector(0,-1){10}}
\put(63,215){\vector(-1,0){10}}
\put(67,215){\vector(1,0){10}}
\put(66.41,216.41){\vector(1,1){10}}
\put(63.59,213.59){\vector(-1,-1){10}}
\multiput(62,218)(-2.2,2.2){5}{\circle*{2}}
\multiput(68,212)(2.2,-2.2){5}{\circle*{2}}

\put(100,250){\circle{4}}
\put(100,252){\vector(0,1){10}}
\put(100,248){\vector(0,-1){10}}
\put(98,250){\vector(-1,0){10}}
\put(102,250){\vector(1,0){10}}
\put(101.41,251.41){\vector(1,1){10}}
\put(98.59,248.59){\vector(-1,-1){10}}
\multiput(97,253)(-2.2,2.2){5}{\circle*{2}}
\multiput(103,247)(2.2,-2.2){5}{\circle*{2}}

\put(220,130){\line(1,1){120}}
\multiput(330,130)(40,0){2}{\line(-1,1){120}}
\put(350,170){$S(C_2)$}

\put(305,175){\circle*{4}}
\put(303.59,176.41){\vector(-1,1){10}}
\put(306.41,173.59){\vector(1,-1){10}}

\put(265,215){\circle{4}}
\put(263.59,216.41){\vector(-1,1){10}}
\put(266.41,213.59){\vector(1,-1){10}}


\end{picture}
\caption{Example \ref{ex:3lines}}
\label{fig:ex}
\end{figure}

\begin{remark}
The above example shows that 
our decomposition is not always a Whitney 
stratification. Indeed, 
$\dim S(C_1)=3$ and 
$\cl_{\M(\A)}(S(C_1))\setminus S(C_1)=C_2\cup C_3$. 
However the subset $C_2\cup C_3$ is not a union of 
our $2$-dimensional components $S(C_4)$ and $S(C_5)$. 
\end{remark}

\begin{lemma}
\label{lem:real}
The real part $\Real S(C)= 
\{\Real(z)\in\Bbb{R}^\ell\mid z\in S(C)\}$ 
of $S(C)$ 
coincides with $\widetilde{C}$. 
\end{lemma}

\begin{proof}
Assume $x\in\Real S(C)$. Then there exists 
$v\in\Bbb{R}^\ell$ such that $x+\sqrt{-1}v\in S(C)$. 
Let $H\in\A_{[X_C]}$. By definition, 
$v\in \tau(X_C)\subset\tau(H)$, and 
$H\notin\Sep(p_C, x)$. Hence $p_C$ and $x$ is not 
separated by any hyperplane $H$ belonging to $\A_{[X_C]}$, we have 
$\Real S(C)\subset\widetilde{C}$. 

Conversely, assume $x\in\widetilde{C}$. 
Since $x$ and $p_C$ are contained in the same 
chamber of $\A_{[X_C]}$, 
we have $\Sep(p_C, x)\cap\A_{[X_C]}
=\emptyset$. Choose $v\in\tau(X_C)
\setminus \bigcup_{H\in\A\setminus\A_{[X_C]}} \tau(H)$. 
Then $x+\sqrt{-1}v\in S(C)$. 
\end{proof}

\begin{lemma}
If $C\in\ch^q_\F(\A)$, then 
$S(C)$ is a contractible 
$(2\ell-q)$-dimensional manifold. 
\end{lemma}

\begin{proof}
Let us prove that $S(C)$ is star-shaped. 
For a point $x+\sqrt{-1}v\in S(C)$ and consider 
the path $p(t)$ with parameter $0\leq t\leq 1$, 
$$
p(t)=(1-t)p_C+t(x+\sqrt{-1}v)=((1-t)p_C+tx)+\sqrt{-1}tv. 
$$
We have $p(1)=x+\sqrt{-1}v$ and $p(0)=p_C$. 
It suffices to prove that 
$p(t)\in\M(\A)$ for $0\leq t\leq 1$. 
If $\Real p(t)=(1-t)p_C+tx\notin H$, 
then obviously we have 
$p(t)\notin H_\Bbb{C}$. 
Suppose $(1-t)p_C+tx\in H$ for some $t$ with $0< t\leq 1$. 
Then by assumption, $H\in\Sep(p_C, x)$. 
By the definition of $S(C)$, 
$v$ is transversal to $H$, so is $tv$, 
which means $p(t)\in\M(\A)$. 
Hence $S(C)$ is star-shaped. 
\end{proof}

Now we have the following : 

\begin{theorem} 
\label{thm:main}
The complement 
$\M(\A)$ of $\A$ is a disjoint union of 
$S(C)$, $C\in\ch(\A)$. Namely, 
$$ 
\M ( \A ) 
= \bigsqcup _{C \in \ch (\A)} S(C). 
$$ 
\end{theorem} 

\begin{proof}
First we prove that $S(C)\cap S(C')=\emptyset$ when 
$C\neq C'$. Suppose this is not the case. Then 
there exists a point 
$x+\sqrt{-1}v\in S(C)\cap S(C')$. 
\begin{itemize}
\item[(a)] 
If both $C, C'$ are 
in $\ch^{q}_{\F}(\A)$ and 
$X_C\neq X_{C'}$, then 
we may assume $
h_{q} (X_C \cap \F^{q}) 
<
h_{q} (X_{C^\prime} \cap \F^{q})$. 
From Lemma \ref{lem:order} we have $C'\not\preceq C$. 
By Lemma \ref{lem:key2} (2), there exists 
$H\in\A_{X_{C'}}\cap\Sep(C, C')$. 
By definition of $S(C')$, $x+\sqrt{-1}v\in S(C')$ 
implies that  
\begin{equation}
\label{eq:sep}
\A_{X_{C'}}\cap\Sep(x, p_{C'})=\emptyset, 
\end{equation}
and 
\begin{equation}
\label{eq:v}
v\in \tau(X_{C'}). 
\end{equation}
It follows from (\ref{eq:sep}) that $x$ and 
$p_C$ are separated by $H$, and from 
(\ref{eq:v}) that $v\in\tau(H)$. 
(Note that $\tau(X_{C'})\subset\tau(H)$.) 
Then we have $x+\sqrt{-1}v\notin S(C)$, which 
contradicts the assumption, this concludes 
$S(C)\cap S(C')=\emptyset$. 

\item[(b)] 
Next we consider the case $C$ and $C'$ are 
in $\ch^{q}_{\F}(\A)$ and $X_C = X_{C'}$. 
By Lemma \ref{lem:key1}, $C$ and $C'$ are 
separated by a hyperplane $H\in\A_{X_{C}}$. 
This implies that $H$ separates 
$\widetilde{C}$ and $\widetilde{C'}$. 
By Lemma \ref{lem:real}, we have 
$\Real S(C)\cap \Real S(C')=\emptyset$. 

\item[(c)] 
Finally, we consider the case 
$C\in\ch^{q}_{\F}(\A)$ and 
$C'\in\ch^{q'}_{\F}(\A)$, 
with $q<q'$. Then again by Lemma \ref{lem:key2} (3), 
there exists a hyperplane $H\in\A_{X_{C'}}$ 
separating $C$ and $C'$. 
As in the case (a), we obtain $x+\sqrt{-1}v\notin S(C)$. 
Therefore $S(C)\cap S(C')=\emptyset$. 
\end{itemize}

Next we prove that 
$$
\M(\A)=
\bigcup_{C\in\ch(\A)} S(C). 
$$
Let $x+\sqrt{-1}v\in\M(\A)$. 
Recall that $\A_{[v]}$ is the set of all hyperplanes 
parallel to $v$, namely, 
$\A_{[v]}=\{H\in\A\mid \tau(H)\ni v\}$. 
Since $v$ is parallel to hyperplanes in 
$\A_{[v]}$, $x$ is not contained in 
$H\in\A_{[v]}$. We can choose 
a chamber $D\in\ch(\A_{[v]})$ 
such that $x\in D$. 
Let $q=\min\{i\mid D\cap\F^i\neq\emptyset\}$. 
Since the closure 
$\cl_{\F^q}(D)$ is a convex 
polytope in $\F^q$ which does not intersect with 
$\F^{q-1}$, there exists the unique 
point $p\in \cl_{\F^q}(D)$ of 
the minimum with respect to 
$h_{q}$. We can choose $X\in L(\A)$ such that 
$p=X\cap\F^q$. 
Note that 
$X=\bigcap_{H\in\A_p} H$ and then 
$v\in\tau(X)$. 
There exists $C\in\ch^q_\F(\A)$ 
satisfying $X_C=X$ and $C\subset D$. 
We prove that 
$$
x+\sqrt{-1}v\in S(C).
$$

It is enough to prove $v\notin\tau(H)$ for 
$H\in\Sep(x, p_C)$. 
Note that $x$ and $p_C$ are contained in the 
same chamber $D\in\ch(\A_{[v]})$. 
Hence if $H\in\Sep(x, p_C)$, then 
$H\notin\A_{[v]}$. By definition of 
$\A_{[v]}$, $v\notin \tau(H)$. 
Therefore $v$ is transverse to $H$, which means 
that $x+\sqrt{-1}v\in S(C)$. 
\end{proof}

\section{Basis of BM-homology}

In this section, we shall prove that the 
closures $\{\cl_{\M(\A)}(S(C))\}_{C\in\ch_\F^q(\A)}$ 
form a basis of $H_{2\ell-q}^{BM}(\M(\A), \Bbb{Z})$. 
In \S\ref{subsec:ori}, we determine orientations 
on our spaces. In \S\ref{subsec:min}, we recall 
the constructions of a basis $\{[\sigma_C]\mid 
C\in\ch_\F^q(\A)\}$ 
of $H_q(\M(\A), \Bbb{Z})$ 
from \cite{Y}. 
By computing intersection numbers of 
$\cl_\M(S(C))$ and 
$[\sigma_{C'}]$, in \S\ref{subsec:min}, we 
prove the main result.

\subsection{Orientations}
\label{subsec:ori}

In this section, we shall define orientations 
for $\F^q$, $X_C$ and $S(C)$ by 
choosing ordered basis of the tangent spaces. 
(See chapter 3 of \cite{gui-pol} for generalities 
of orientations and intersections of manifolds.) 

Recall that the subspace $\F^q$ is defined by 
$\{x\in\Bbb{R}^\ell\mid h_{q+1}(x)=\dots=h_\ell(x)=0\}$, 
where $h_i (i=1, \dots, \ell)$ are linear forms. 
Hence $(h_1, \dots, h_q)$ forms a coordinate 
of the space $\F^q$. We consider 
the orientation defined by the ordered basis 
$(\partial_{h_1}, \dots, \partial_{h_q})$ 
of $T_x\F^q=\tau(\F^q)$. 
In particular, the orientation of 
$\Bbb{R}^\ell$ is determined by 
the ordered basis 
$(\partial_{h_1}, \dots, \partial_{h_\ell})$. 
If $C$ belongs to $\ch^q_\F(\A)$, then $X_C$ is an 
affine subspace complemental to $\F^q$. 
So $(h_{q+1}, \dots, h_\ell)$ forms a coordinate 
of $X_C$, and we consider the 
orientation determined by the dual basis 
$(\partial_{h_{q+1}}, \dots, 
\partial_{h_\ell})$ with an order. Note that 
the intersection number 
$\F^q\cdot X_C=(-1)^{q(\ell-q)}\cdot X_C\cdot \F^q$ equals to $+1$. 

Next we consider the orientation of $S(C)$. 
By definition, the tangent space of $S(C)$ at $p_C$ 
is expressed as 
$$
T_{p_C}S(C)\simeq T_{p_C}C\oplus\sqrt{-1}\cdot\tau(X_C). 
$$
Thus we define the orientation by 
$(\partial_{h_1}, 
\partial_{h_2}, 
\dots, 
\partial_{h_\ell}, 
\sqrt{-1}\partial_{h_{q+1}}, 
\dots, 
\sqrt{-1}\partial_{h_\ell})$. 
The case $q=0$ defines an orientation on $\Bbb{C}^\ell$ 
by 
$(\partial_{h_1}, 
\dots, 
\partial_{h_\ell}, 
\sqrt{-1}\partial_{h_{1}}, 
\dots, 
\sqrt{-1}\partial_{h_\ell})$. 
We should note that this orientation 
is different from the usual 
one defined by 
$(\partial_{h_1}, 
\sqrt{-1}\partial_{h_{1}}, 
\partial_{h_2}, 
\sqrt{-1}\partial_{h_{2}}, 
\dots, 
\partial_{h_\ell}, 
\sqrt{-1}\partial_{h_\ell})$.

The rest will be used in \S\ref{sec:os}. 
Let $I=\{i_1, \ldots, i_q\}\subset\{1, \ldots, n\}$ be 
an ordered subset of $q$ indices, 
$\A(I):=\{H_{i_1}, \ldots, H_{i_q}\}$ be 
a subarrangement consisting of $q$ hyperplanes. 
Assume $H_{i_1}, \ldots, H_{i_q}$ are independent, that is, 
$d\alpha_{i_1}\wedge\cdots\wedge d\alpha_{i_\ell}\neq 0$ 
or equivalently, the intersection 
$X(I):=H_{i_1}\cap\cdots\cap H_{i_q}$ is a non-empty 
subspace of codimension $q$. 

\begin{dflem}
The set of chambers $\ch(\A(I))$ consists of 
$2^q$ chambers. There is a unique chamber, denoted by 
$C_0(I)\in\ch(\A(I))$, which satisfies 
$C_0(I)\cap\F^{q-1}=\emptyset$. 
\end{dflem}
\begin{proof}
The Poincar\'e polynomial 
of $\Bbb{C}^\ell\setminus\bigcup_{i\in I}H_{i,\Bbb{C}}$ 
is $(1+t)^q$. In particular, $b_q=1$. 
Hence by Proposition \ref{prop:betti}, 
$|\ch_{\F}^q(\A(I))|=1$. 
\end{proof}

Choose a normal vector $w_{i_k}\perp H_{i_k}$ for 
each $H_{i_k}$ such that $C_0(I)$ is contained in 
the half space $H_{i_k}+\Bbb{R}_{>0}\cdot w_{i_k}$. 
Suppose $H_{i_1}, \dots, H_{i_q}$ are independent (i.e., 
the intersection $X(I)=H_{i_1}\cap\dots\cap H_{i_q}$ 
has codimension $q$ with $q\leq\ell$). 
Since $\F^q$ is generic, $\F^q\cap X(I)$ is $0$-dimensional. 
Thus by the identification 
$\Bbb{R}^\ell/X(I)\simeq\F^q$,  
the normal vectors $w_{i_1},\dots, w_{i_q}$ induce 
a basis of $\F^q$. 

\begin{definition}
For an ordered $q$-tuple 
$I=\{i_1, \ldots, i_q\}\subset\{1, \ldots, n\}$, define 
$\varepsilon(I)$ by 
$$
\varepsilon(I)=
\left\{
\begin{array}{ll}
0& \mbox{if $H_{i_1}, \ldots, H_{i_q}$ are dependent,} \\
1& \mbox{if $(w_{i_1}, \ldots, w_{i_q})$ induces a 
positive basis of $\F^q$,} \\
-1&\mbox{if $(w_{i_1}, \ldots, w_{i_q})$ induces a 
negative basis of $\F^q$.}  
\end{array}
\right.
$$
\end{definition}

\subsection{Minimal CW-decomposition}
\label{subsec:min}

Here we recall results from \cite[\S 5.2]{Y}. 
For each $C\in\ch_\F^q(\A)$, 
there exists a continuous map, unique up to homotopy, 
$$
\sigma_C:
(D^q, \partial D^q)
\longrightarrow 
(\M^q, \M^{q-1}), 
$$
from the $q$-dimensional disk to the 
complement $\M^q=\M(\A)\cap\F^q_\Bbb{C}$ 
such that 
\begin{quote}
(Transversality) 
$\sigma_C(0)=p_C\in C\cap\F^q$, and 
$\sigma_C(D^q)$ intersects $C\cap\F^q$ 
transversally in $\F^q_\Bbb{C}$ at the point; 
$\sigma_C(D^q)\pitchfork C=\{p_C\}$, and \\
(Non-intersecting) 
$\sigma_C(D^q)\cap C'=\emptyset$ for 
$C'\in\ch_\F^q(\A)\backslash\{C\}$. 
\end{quote}
These properties guarantee the following homotopy 
equivalence (\cite[4.3.1]{Y}): 
\begin{equation}
\M^q\simeq\M^{q-1}\cup_{(\partial\sigma_C)}
\left(
\bigsqcup_{C\in\ch_\F^q(\A)}D^q
\right), 
\end{equation}
where the right hand side is obtained by 
attaching $q$-dimensional disks to $\M^{q-1}$ 
along 
$\partial\sigma_C:\partial D^q\rightarrow \M^{q-1}$ 
for ${C\in\ch_\F^q(\A)}$. 

Recall that $T_{p_C}\M^q\simeq 
\tau(\F^q)\oplus\sqrt{-1}\cdot
\tau(\F^q)$. We introduce an orientation 
on $\sigma_C$ by identifying 
$T_{p_C}\sigma_C(D^q)$ with 
$\sqrt{-1}\cdot\tau(\F^q)$, 
equivalently, by an ordered basis 
$(\sqrt{-1}\partial_{h_1}, \dots, \sqrt{-1}\partial_{h_q})$. 

\begin{prop}[\cite{Y}]
\label{prop:trubasis}
(1) $[\sigma_C]\in H_q(\M^q, \M^{q-1}; \Bbb{Z}), (C\in
\ch^q_\F(\A))$ forms a basis. 

(2) $H_q(\M^q, \Bbb{Z})\simeq 
H_q(\M^q, \M^{q-1}; \Bbb{Z})\simeq 
H_q(\M(\A), \Bbb{Z})$. 
\end{prop}

We construct the basis of $H_{2\ell-q}^{BM}(\M(\A), \Bbb{Z})$. 
Let $C\in\ch_\F^q(\A)$. 
Lemma \ref{lem:real} indicates that 
\begin{equation}
\label{eq:cyc}
\cl_{\M(\A)}(S(C))=
(\widetilde{C}\times\sqrt{-1}\cdot\tau(X_C))\cap\M(\A), 
\end{equation}
which is a closed oriented $(2\ell-q)$-dimensional 
submanifold of $\M(\A)$ because $\dim X_C=\ell-q$. 
The closed submanifold $\cl_{\M(\A)}(S(C))$ determines a 
cycle 
$[\cl_{\M(\A)}(S(C))]\in H_{2\ell-q}^{BM}({\M(\A)}, \Bbb{Z})$. 

\begin{theorem}
\label{thm:basis}
The classes 
$\{[\cl_{\M(\A)}(S(C))]\}_{C\in\ch_\F^q(\A)}$ form 
a basis of the $(2\ell-q)$-th Borel-Moore homology group 
$H_{2\ell-q}^{BM}(\M(\A), \Bbb{Z})$. 
\end{theorem}

\begin{proof}
We compute the intersection number of 
$[\cl_{\M(\A)}(S(C))]\in H_{2\ell -q}^{BM}({\M(\A)})$ and 
$[\sigma(C')]\in H_{q}({\M(\A)})$, and show that 
the intersection matrix 
$$
I([\cl_{\M(\A)}(S(C))], [\sigma(C')])_{C, C'\in\ch_\F^q(\A)}
$$ 
is a triangular matrix with each diagonal entry 
$(-1)^{q(\ell-q)}$. 

We fix an ordering on $\{C_1, \dots, C_b\}=\ch_\F^q(\A)$ 
in such a way that $C_i\preceq{C_j}\Longrightarrow i<j$ 
(e.g. choose an ordering with 
$h_q(X_{C_1}\cap\F^q)\leq h_q(X_{C_2}\cap\F^q)\leq
\cdots\leq h_q(X_{C_b}\cap\F^q)$.) Since 
$\F^q$ and $X_C$ are mutually complementary in $\Bbb{R}^\ell$, 
the tangent 
space $T_{p_C}\Bbb{C}^\ell$ can expressed as 
$$
T_{p_C}\Bbb{C}^\ell=
T_{p_C}\Bbb{R}^\ell\oplus
\sqrt{-1}\cdot
T_{p_C}\F^q
\oplus\sqrt{-1}\cdot
\tau(X_C). 
$$
The above mentioned properties and 
(\ref{eq:cyc}) implies that $\cl_{\M(\A)}(S(C))$ 
intersects transversally to $\sigma_{C'}$ if 
and only if $p_{C'}\in\widetilde{C}$. In fact, 
we have $T_{p_C}\cl_{\M(\A)}(S(C))=\Bbb{R}^\ell
\oplus\sqrt{-1}\cdot\tau(X_C)$ 
and $T_{p_{C'}}\sigma_{C'}(D^q)=
\sqrt{-1}\cdot T_{p_{C'}}\F^q$, which implies the 
transversality and its 
intersection number is 
$(-1)^{q(\ell-q)}$. 
\end{proof}

\section{Relations with OS-type generators}
\label{sec:os}

As is mentioned in \S \ref{sec:intro}, there is a 
canonical isomorphism 
$\varphi: H^q(\M(\A), \Bbb{Z})\stackrel{\cong}{\longrightarrow}
H^{BM}_{2\ell-q}(\M(\A), \Bbb{Z})$ between cohomology and 
Borel-Moore 
homology of $\M(\A)$. 
In this section, we describe $\varphi$ explicitly 
by using the basis introduced in the previous sections. 

First note that both 
$H^{BM}_{2\ell-q}(\M(\A), \Bbb{Z})$ and 
$H^q(\M(\A), \Bbb{Z})$ are dual to the 
homology group $H_q(\M(\A), \Bbb{Z})$. 
The pairing 
$H^{BM}_{2\ell-q}(\M(\A), \Bbb{Z})\times
H_q(\M(\A), \Bbb{Z})\rightarrow\Bbb{Z}$ is defined 
by the intersection $I(\cdot, \cdot)$, and 
$H^q(\M(\A), \Bbb{Z})\times
H_q(\M(\A), \Bbb{Z})\rightarrow\Bbb{Z}$ is 
defined by the cap product $\cap$ (or the integration 
if we consider de Rham cohomology). 

The structure of cohomology ring 
$H^q(\M(\A), \Bbb{Z})$ is well studied 
(see e.g. \cite{orl-sol}) and especially, by Arnold-Brieskorn's 
result, it is generated by logarithmic forms 
$$
\omega_i=\frac{1}{2\pi\sqrt{-1}}\frac{d\alpha_i}{\alpha_i}, 
$$
for $i=1, \dots, n$. The $q$-th cohomology group 
$H^q(\M(\A), \Bbb{Z})$ is spanned by 
$\omega_{i_1, \dots, i_q}=
\omega_{i_1}\wedge
\omega_{i_2}\wedge\dots\wedge
\omega_{i_q}$ 
with $H_{i_1}, \dots, H_{i_q}$ linearly independent.

\begin{theorem}
Let $I=\{i_1, \dots, i_q\}\subseteq \{1, \dots, n\}$ be 
an ordered index (see \S\ref{subsec:ori} for 
notation). Then 
$$
\varphi(\omega_I)=
(-1)^{q(\ell-q)}\varepsilon(I)\cdot
\sum_C [\cl_\M(S(C))], 
$$
where $C$ runs over all chambers $C\in\ch_\F^q(\A)$ satisfying 
$C\subset C_0(I)$ and $\tau(X_C)=\tau(X(I))$. 
\end{theorem}

\begin{proof}
Let us define $S(I)\subset\Bbb{C}^\ell$ to be 
$$
S(I)=C_0(I)\oplus\sqrt{-1}\cdot\tau(X(I)). 
$$
Then $\cl_{\M(\A)}(S(I))$ is a disjoint union of 
$\cl_\M(S(C))$'s with $C$ running over all chambers 
$C\in\ch_\F^q(\A)$ satisfying 
$C\subset C_0(I)$ and $\tau(X_C)=\tau(X(I))$. 
It is enough to show 
$\varphi(\omega_I)=
(-1)^{q(\ell-q)}\varepsilon(I)\cdot\cl_\M(S(I))$. 
To do this, we shall see the pairing with 
the homology class $[\sigma_{C'}]\in 
H_q(\M^q, \M^{q-1}, \Bbb{Z})\cong H_q({\M(\A)}, \Bbb{Z})$. 

First we compute $\int_{[\sigma_{C'}]}\omega_I$. The 
complement $\Bbb{C}^\ell\setminus\bigcup_{i\in I}
H_{i, \Bbb{C}}$ is homotopy equivalent to 
$(\Bbb{C}^*)^q\simeq (S^1)^q$. The top homology 
$H_q(\Bbb{C}^\ell\setminus\bigcup_{i\in I}
H_{i, \Bbb{C}}, \Bbb{Z})\cong \Bbb{Z}$ is rank one. 
If $C'\subset C_0(I)$, then $[\sigma_{C'}]$ is transversal to 
$C_0(I)$. By applying Proposition \ref{prop:trubasis} to the arrangement 
$\A(I)=\{H_{i_1}, \dots, H_{i_q}\}$, we know the fact that 
$[\sigma_{C'}]$ is a generator of 
$H_q(\Bbb{C}^\ell\setminus\bigcup_{i\in I}
H_{i, \Bbb{C}}, \Bbb{Z})$. Similarly, 
if $C'\not\subset C_0(I)$, then $[\sigma_{C'}]=0$. 
We have 
$$
\int_{[\sigma_{C'}]}\omega_I=
\left\{
\begin{array}{ll}
\varepsilon(I)& \mbox{ if }C'\subset C_0(I),\\
0&\mbox{else}. 
\end{array}
\right.
$$

By a computation similar to the proof of 
Theorem \ref{thm:basis}, we have 
$$
I([S(I)], [\sigma_{C'}])=
\left\{
\begin{array}{ll}
(-1)^{q(\ell-q)}& \mbox{ if }C'\subset C_0(I),\\
0&\mbox{else}. 
\end{array}
\right.
$$
This completes the proof. 
\end{proof}

\begin{remark}
The correspondences between chambers and 
de Rham cohomology groups were investigated by 
Varchenko and Gel'fand in \cite{var-gel}. 
Indeed, the cycle $S(I)$ appeared in their paper. 
\end{remark}

\noindent
Ko-Ki Ito,

RIMS,
Kyoto University,
Sakyo-ku, Kyoto 606-8502, Japan

{\em e-mail address}\ : \  koki@kurims.kyoto-u.ac.jp

\medskip

\noindent
Masahiko Yoshinaga

Department of Mathematics,
Faculty of Science,
Kyoto University,
Sakyo-ku, Kyoto 606-8502, Japan

{\em e-mail address}\ : \ mhyo@math.kyoto-u.ac.jp

\end{document}